\documentclass[11pt,twoside]{article}

\usepackage{latexsym, amssymb, amsmath, amsfonts, theorem, subfigure}
\usepackage{mathptmx}

\theoremstyle{plain} \theorembodyfont{\itshape}
\newtheorem{theorem}{Theorem}[section]
\newtheorem{proposition}[theorem]{Proposition}
\newtheorem{lemma}[theorem]{Lemma}

\theorembodyfont{\rmfamily}
\newtheorem{definition}[theorem]{Definition}

\newtheorem{remark}[theorem]{Remark}
\newtheorem{conjecture}[theorem]{Conjecture}
\newtheorem{problem}[theorem]{Problem}
\newenvironment{proof}
{{\noindent \textbf{Proof}\,\,}}{\hspace*{\fill}$\Box$\medskip}
\numberwithin{equation}{section}

\usepackage{amscd,amsfonts}
\usepackage{pst-all}
\usepackage{amsfonts}
\usepackage[dvips]{graphicx}
\usepackage{epic}
\usepackage{verbatim}
\usepackage{amssymb, upref}
\usepackage[dvips]{graphicx}
\usepackage{epsf}

\def\C{{\mathbb C}}
\def\Z{{\mathbb Z}}
\def\Q{{\mathbb Q}}

\def\V{\text{Var}}
\def\t{{\theta}}

\newcounter{listcount}

\title{\textbf{Monodromy and Tangential Center Problems }\thanks{The authors would like to thank the Universities of Bourgogne and
Plymouth respectively for their kind hospitality during the preparation of this work.}}
\author{\sc{Colin Christopher},\\ School of Mathematics and Statistics,\\
University of Plymouth,\\
Plymouth PL4 8AA, U.K.\\ \\
\sc{Pavao Marde\v si\'c},\\
Institut de Math\'ematiques de Bourgogne, \\Unit\'e mixte de recherche 5584 du C.N.R.S.,\\
Universit\'e de Bourgogne, \\
B.P. 47 870, 21078 Dijon Cedex, France.}

\begin{document}
\maketitle

\begin{abstract}
We consider families of Abelian integrals arising from perturbations
of planar Hamiltonian systems.  The {\it tangential center focus
problem} asks for the conditions under which these integrals vanish
identically. The problem is closely related to the {\it monodromy
problem}, which asks when the monodromy of a vanishing cycle
generates the whole homology of the level curves of the Hamiltonian.
We solve both these questions for the case when the Hamiltonian is
hyperelliptic.  As a side-product, we solve the corresponding
problems for the ``$0$-dimensional Abelian integrals" defined by
Gavrilov and Movasati.
\end{abstract}


\section{Introduction}
The weak Hilbert $16^{th}$ problem, as posed by Arnold \cite{A},
asks:

\begin{problem}[Weak Hilbert $16^{th}$ problem]\label{16}
Let $F\in\C[x,y]$ and $\omega=P(x,y)\,dx+Q(x,y)\,dy$, with $P,Q\in
\C[x,y]$ and consider the system
\begin{equation}\label{def}
dF+\varepsilon\omega=0.
\end{equation}
Bound the number of real limit cycles in the system (\ref{def}) for small
values of $\varepsilon$. \end{problem}

The problem leads to the study of the zeros of the Abelian
integral
\begin{equation}\label{ab}
I(t)=\int_{\delta(t)}\omega,
\end{equation}
where $\delta(t)$ is a family of cycles lying in $F^{-1}(t)$.
Provided this integral does not vanish identically, limit cycles
of \eqref{def} correspond to zeros of $I(t)$ for generic values
of $t$.

That is, to first order we are led to solve the following
simpler problem.

\begin{problem}[Tangential Hilbert $16^{th}$
problem]\label{tangential}
  Bound the number of zeros of the Abelian integral \eqref{ab}
  in terms of the degrees of $F$ and $\omega$.
\end{problem}

When, the Abelian integral vanishes identically, it provides no
information about the limit cycles of \eqref{def}, and higher
order perturbation theory must be used.  It is therefore of interest
to understand under what conditions this can happen.

The classical center focus problem asks for a characterization of centers
of planar polynomial vector fields. The problem of when an Abelian
integral vanishes identically along a vanishing cycle, can be seen as a
tangential version of this problem.

\begin{problem}[Tangential center focus problem]\label{tangentialcf}
Characterize the conditions under which the Abelian integral
$I(t)$ of (\ref{ab}) vanishes identically along a vanishing cycle $\delta(t)$
associated to a Morse singular point $p$ of $F$.
\end{problem}

If such a $\delta(t)$ exists, we say that \eqref{def} has a
{\it tangential center} at $p$.
\smallskip

Problem~\ref{tangentialcf} was solved by Il'yashenko for generic $F$ by proving
that, for generic $F$, the monodromy acts transitively on the
first homology group of the generic fiber. The vanishing of $I(t)$
therefore implies the vanishing of the Abelian integral along all
cycles in $H_1(F^{-1}(t)).$ This in turn implies that the form
$\omega$ is relatively exact.

In fact, the condition that the vanishing of an Abelian integral
\eqref{ab} along a family of cycles $\delta(t)$ implies the relative
exactness of $\omega$ is called ``condition $(*)$" by Fran\c{c}oise.
Under condition $(*)$, Fran\c{c}oise \cite{F} (see also \cite{Y})
gives an algorithm
for calculating higher order terms of the displacement function.

By the results of Bonnet and Dimca \cite{BD} (and, in a more
restricted setting, Gavrilov \cite{G} and Il'yashenko \cite{Il}),
if we assume the vanishing of the Abelian integrals on all cycles,
then $P(F)\omega$ must be relatively exact, for some polynomial $P$,
whose roots correspond to some exceptional fibers. Condition $(*)$
therefore follows automatically (after possible multiplication of
$\omega$ by a factor $P(F)$) if we can show that under the action of
the monodromy, the cycle $\delta(t)$ generates the whole of the
homology of the generic fiber of $F$ over $\Q$.

This leads to a natural problem:

\begin{problem}[Monodromy problem]\label{monprob}  Under what conditions on
$F$ is the $\Q$-subspace of $H_1(F^{-1}(t),\Q)$ generated by the
images of a vanishing cycle of a Morse point under monodromy, equal
to the whole of $H_1(F^{-1}(t),\Q)$?
\end{problem}

The principal motivation for this paper was to solve these last two
problems in the case when $F(x,y)=y^2+f(x)$ (the {\it hyperelliptic
case}).  In more detail, we prove the following two theorems.

\begin{theorem}\label{hyperellipticthm}
   The system \eqref{def}, with $F = y^2 + f(x)$, has a tangential center
   with associated vanishing cycle $\delta(t)$, if and only if (i) or (ii) is
   verified:
   \begin{list}{(\roman{listcount})\hfill}{\usecounter{listcount} \setlength{\leftmargin}{4em} \setlength{\labelwidth}{2em}}
   \item  the form $\omega$ is relatively exact i.e.
   $\omega=A\,dF+dB$, with $A,B\in\C[x,y]$.
   \item  $f$ is decomposable i.e.
           $f=g\circ h$, and $\omega=\tilde{\omega}+\pi^*\eta$, where $\tilde{\omega}$ is
           relatively exact, and $\pi_*\delta(t)$ is homotopic to zero
           in $y^2+g(z)=t$, where $\pi(x,y)=(h(x),y)=(z,y)$.
   \end{list}
\end{theorem}

\begin{theorem}\label{hyperellipticmon}
   Let $F = y^2 + f(x)$,
   with associated vanishing cycle $\delta(t)$ at a Morse point,
   then one of the following must hold.
   \begin{list}{(\roman{listcount})\hfill}{\usecounter{listcount} \setlength{\leftmargin}{4em} \setlength{\labelwidth}{2em}}
   \item  the monodromy of $\delta(t)$ generates the homology $H_1(F^{-1}(t),\Q)$.
   \item  $f$ is decomposable i.e.
   $f=g\circ h$, and $\pi_*\delta(t)$ is homotopic to zero
           in $y^2+g(z)=t$, where $\pi(x,y)=(h(x),y)=(z,y)$.
   \end{list}
\end{theorem}

\medskip

To prove the above theorems we first reduce them to analogous $0$-dimensional
problems which we consider next.

We define a $0$-dimensional Abelian integral following
Gavrilov and Movasati \cite{GM}.

Let $f\in\C[x]$ be a polynomial and $\delta(t)\in H_0(f^{-1}(t))$ a $0$-cycle:
that is, $\delta(t)=\sum n_i x_i(t) \in f^{-1}(t)$, $n_i\in \C$, with
$\sum n_i=0$ and let $\omega\in\C[x]$ be a polynomial ($0$-form).
A {\it $0$-dimensional Abelian integral} is given by a function
\begin{equation}\label{int0}
I_0(t) = \int_{\delta(t)}\omega:=\sum n_i \omega(x_i(t)).
\end{equation}

A cycle of the form $\delta(t)=x_i(t)-x_j(t)$,
with $f(x_i(t))=f(x_j(t))=t$ is called a {\it simple cycle}.

We characterize the vanishing of $0$-dimensional
Abelian integrals along simple cycles (the {\it $0$-dimensional
tangential center focus problem}) and the conditions under
which a simple cycle generates the whole of the
reduced homology $H_0(f^{-1}(t))$ of the generic fiber (the {\it $0$-dimensional
monodromy problem}).

\begin{theorem}\label{center}
Let $f,\omega\in\C[x]$, $\delta(t)=x_i(t)-x_j(t)$ be a simple cycle
in the generic fiber of $f$. The Abelian integral
$I(t)=\int_{\delta(t)}\omega$ vanishes identically if and only if
there exists a polynomial $h$ with $\deg(h)>1$ such that $f=g\circ h$ and
$\omega=\eta\circ h$, for some polynomials $g$ and $\eta$, and
$\delta(t)=\tilde{\delta}(h(t))$ for some simple cycle
$\tilde{\delta}$ of $g$.
\end{theorem}

\begin{theorem}\label{0dimmon}
Let $\delta(t)=x_i(t)-x_j(t)$ be a simple cycle
in the generic fiber of $f$. Then either
\begin{list}{(\roman{listcount})\hfill}{\usecounter{listcount} \setlength{\leftmargin}{4em} \setlength{\labelwidth}{2em}}
\item The cycle $\delta(t)$ generates the reduced homology $H_0(f^{-1}(t)))$.
\item $f$ decomposes as $f=g\circ h$, $(\deg(h)>1)$, and
$\delta(t)=\tilde{\delta}(h(t))$ for some simple cycle
$\tilde{\delta}$ of $g$.
\end{list}
\end{theorem}

\medskip
The principal tools in the proof of these theorems is L\"uroth's theorem on
field extensions and the Burnside-Schur theorem on group actions with
a regular cyclic subgroup.  We recall both these theorems in Section~\ref{prelims} below.

\begin{remark}
  If a cycle $\delta(t)$ is not simple, then the theorems above do
  not hold.  A counter-example is provided if $f(x)=T_p(x)$, a Chebyshev
  polynomial of prime degree.  We examine this case in detail in the
  final section.

  Similarly, the polynomial $F(x,y) = y^2 + T_p(x)$ gives a counter-example
  to a generalization of Theorems \ref{hyperellipticthm} and \ref{hyperellipticmon}
\end{remark}

\section{Preliminaries}\label{prelims}

We recall some definitions from group theory

\begin{definition}\rm
\
\begin{enumerate}
\item
  Let $G$ be a group acting on a finite set $S$.  We say that the action is
  {\it imprimitive} if there exists a non-trivial decomposition of $S$, $S = \bigcup S_i$,
  such that for each element of $g$ and each $i$, $g$ sends $S_i$ into $S_j$ for some $j$.  The
  action is called {\it primitive} if it is not imprimitive.
\item
  An action is {\it transitive} if given any pair of elements of $S$,
  $s_1$ and $s_2$, there is an element $g\in G$ which sends $s_1$ to
  $s_2$.
\item
  An action is {\it 2-transitive} if given any two pairs of elements of $S$,
  $(s_1,s_2)$ and $(s_3,s_4)$, there is an element $g\in G$ which sends $s_1$ to
  $s_3$ and $s_2$ to $s_4$.
\item
  An action is {\it regular} if given two elements $s_1$ and $s_2$ of $S$ there is
  a unique element $g$ of $G$ which sends $s_1$ to $s_2$.
\item
  Given $s\in S$, we denote the group of all elements of $G$ which fix $s$ (the
  {\it stabilizer} of $s$) by $G_s$.
\end{enumerate}
\end{definition}

The following theorem is classical, but we state it here for convenience.

\begin{theorem}[L\"uroth]\label{luroth}
Let $k(t)$ be a transcendental extension of a field $k$.  Any subfield $K\subset k(t)$,
such that $k\varsubsetneq K$, is of the form $K = k(r)$ for some $r\in k(t)$.
\end{theorem}

\begin{proposition}\label{tower4}
  Let $G$ be a group acting transitively on a finite set $S$.  The action of $G$ on $S$
  is imprimitive if and only if for some element $s$ of $S$ there is a subgroup $H$ of
  $G$ such that
  \begin{equation}\label{tower99}
    G_s \varsubsetneq H \varsubsetneq G,
  \end{equation}
  where $G_s$ is the subgroup of $G$ of all elements which leave $s$ fixed.
\end{proposition}

\proof  Suppose that the action of $G$ on $S$ is imprimitive, and let $S_0$ be
the subset which contains $s$ in the decomposition of $S$.  We let $H$ be the
subset of $G$ consisting of all elements which fix $S_0$.  Since $S_0$ is non-trivial
it must have more than one element but be strictly contained in $S$. From the transitivity
of $G$, $H$ must be therefore strictly larger than $G_s$, but smaller than $G$.

Conversely, if \eqref{tower99} holds, we can consider the orbit of $s$ under the action
of $H$: call this $S_1$.  This cannot be the whole of $S$, or else $H$ would be
the same as $G$ (since $H$ already contains $G_s$).  However, it must contain more
elements than just $s$.  Now consider the action of $G$ on $S_1$.
If $s' \in g_1(S_1)\cap g_2(S_1)$ then there exist some $h_1,h_2\in H$ such that $g_1h_1(s)=s'=g_2h_2(s)$.
Thus $h_2^{-1}g_2^{-1}g_1h_1 \in G_s$, and hence $g_2^{-1}g_1 \in H$ and $g_1(S_1)=g_2(S_1)$.
Therefore the images of $S$ under $G$ give a partition of $S$ on which $G$ acts imprimitively.

\medskip

Recall that the affine group $\mbox{Aff}(\Z_p)$ is the group of all affine transformations
of $\Z_p$ to itself.  That is, it is the group of all maps from $\Z_p$ to itself of the form
$x\mapsto ax+b$ for $a$, $b \in \Z_p$ with multiplication given by composition.  Note that
every element of $\mbox{Aff}(\Z_p)$ fixes at most one element of $\Z_p$.  We will use this
fact in the proof of Theorem~\ref{muller}

\begin{theorem}[Burnside-Schur]\label{BS}
   Every primitive finite permutation group containing a regular cyclic subgroup is
   either 2-transitive or permutationally isomorphic to a subgroup of the affine group
   $\mbox{Aff}(\Z_p)$, where $p$ is a prime.
\end{theorem}

\proof  See \cite{DM} or \cite{EP}.

\section{Monodromy groups of polynomials}\label{monpoly}

Let $f(x)$ be a polynomial of degree $n>0$, and consider the solutions, $x_i(t)$, of
the equation $f(x)=t$.  Let $\Sigma$ be the set of critical points $t \in \C$ for which
$f(x)=t$ and $f'(x)=0$ have a common solution.  Clearly there are at most
$n(n-1)$ of these points.  As $t$ takes values in $\C\setminus\Sigma$ the functions
$x_i(t)$ are well-defined.  The group $G=\pi_1(\C\setminus\Sigma)$ acts on the $x_i(t)$.
The action is always transitive (Proposition~\ref{BS2}).

\begin{definition}\rm
   Let $G$ be as above, then the action of G on the set of $x_i$ is called
   the {\it monodromy} group of the polynomial $f$, denoted $\mbox{Mon}(f)$.
\end{definition}

\begin{proposition}\label{BS2}
   Let $f$ be a polynomial over $\C$ of degree $n$, then its monodromy group,
   $\mbox{Mon}(f)$, is transitive and has a cyclic subgroup of degree $n$
   which acts regularly on the roots of $f$.
\end{proposition}

\proof  The first statement follows from the second.
When $t$ is large the $x_i$ can be expanded as
\[
     x_i = \omega^r t^{1/n} + O(t^{(1/n)-1}),
\]
where $\omega$ is an $n$-th root of unity.  Thus, taking a
sufficiently large loop in $\C\setminus\Sigma$, we obtain an element of
$G$ which is an $n$-cycle.  This element generates a cyclic
subgroup of $G$ which acts regularly on the roots of $f(x)=t$.

\medskip

Elements of the monodromy group clearly lie in the Galois group
of $f(x)-t=0$ over $\C(t)$.  The following fundamental theorem \cite{For}
states that all elements of the Galois group can be generated in this
way.

\begin{theorem}\label{galois}
   The monodromy group of $f$, $\mbox{Mon}(f)$, is isomorphic to the
   Galois group of $f(x)-t$ considered as a polynomial over $\C(t)$.
\end{theorem}

\begin{definition}\rm
   We say that a polynomial $f(x)$ is {\it decomposable} if and only if there exist
   two polynomials $g$ and $h$, both of degree greater than one, such that
   $f(x) = g(h(x))$.
\end{definition}

\begin{lemma}\label{lem5.0}
   Suppose that $f(x)$ is a polynomial over $\C$ which can be expressed
   as $g(h(x))$ for $g$ and $h$ rational functions of degree greater than
   one over $\C$, then there is a decomposition $f(x)=\tilde{g}(\tilde{h}(x))$, where
   $\tilde{g}$ and $\tilde{h}$ can be chosen to be polynomials over $\C$.
\end{lemma}

\proof  Let $h(x)=r(x)/s(x)$, where $r$ and $s$ are polynomials over $\C$.
Without loss of generality, if $m$ is a M\"obius transformation, we can
rewrite the decomposition of $f$ as $f = \tilde{g}\circ \tilde{h}$ with $\tilde{g} = g\circ m^{-1}$ and
$\tilde{h}=m\circ h$.
In this way, we can assume that $\tilde{h}=\tilde{r}/\tilde{s}$, with $\deg(\tilde{s})<\deg(\tilde{r})$, and
both $\tilde{r}$ and $\tilde{s}$ monic.  Now,
\[
     \tilde{g}(\tilde{h}(x)) =  \frac {  \prod_{i=1}^q \alpha_i\tilde{r}(x)+\beta_i\tilde{s}(x) }
                      {  \prod_{i=1}^q \gamma_i\tilde{r}(x)+\delta_i\tilde{s}(x)},
\]
for some constants $\alpha_i,\beta_i,\gamma_i,\delta_i\in \C$.
If $\alpha_i\tilde{r}+\beta_i\tilde{s}$ shares a common factor with
$\gamma_j\tilde{r}+\delta_j\tilde{s}$, these two polynomials must be the same
up to a constant multiple, whence we can assume that the fraction above allows no
further cancelations. Since $\tilde{g}\circ\tilde{h}$ is a
polynomial, $\prod\gamma_i\tilde{r}+\delta_i\tilde{s}$ must be a constant, and
hence the denominator has no dependence on $\tilde{r}$, and $\tilde{s}$ must be a
constant (and therefore $\tilde{s}=1$).  The result follows directly.

\begin{proposition}\label{monthms}
   Let $f$ be a polynomial as above and let $G = \mbox{Mon}(f)$ be its monodromy
   group.  Then
   \begin{list}{(\roman{listcount})\hfill}{\usecounter{listcount} \setlength{\leftmargin}{4em} \setlength{\labelwidth}{2em}}
   \item the action of $G$ is imprimitive if and only
   if the polynomial $f$ is decomposable.
   \item the action is
   2-transitive if and only if the divided differences polynomial
   \[\Delta(x,y)=(f(x)-f(y))/(x-y)\] is irreducible.
   \end{list}
\end{proposition}

\proof
\begin{list}{(\roman{listcount})\hfill}{\usecounter{listcount} \setlength{\leftmargin}{4em} \setlength{\labelwidth}{2em}}

\item Let $t\in\C\setminus\Sigma$, let $s$ be a root of $f(x)-t$, and $G_s$ the stabilizer of $s$.
From Proposition~\ref{tower4} we have
\begin{equation}\label{GHG}
    G_s \varsubsetneq H \varsubsetneq G.
\end{equation}
The splitting field of $f(x)-t$ over $\C(t)$ is just $\C(x_1(t),\ldots,x_n(t))$.
Under the Galois correspondence, we have
\begin{equation}\label{tower2.0}
     \C(x_k(t)) \varsupsetneq K \varsupsetneq \C(t),
\end{equation}
where $K$ is the fixed field of $H$, and $x_k(t)$ is the
root of $f(x)=t$ corresponding to $s$.
From L\"uroth's theorem, we must have $K = \C(r(x_k))$, for some rational
function $r$ over $\C$.
Then \eqref{tower2.0} implies that $t = s(r(x_k))$ for some rational
function $s$.  Thus $f(x) = s(r(x))$, and Lemma~\ref{lem5.0} shows
that $s$ and $r$ can in fact be chosen to be polynomials.
Conversely, given a decomposition $f(x)=s(r(x))$, we take $K=\C(r(x))$
and obtain \eqref{GHG} from \eqref{tower2.0}
via the Galois correspondence.
\item  Let $y=x_1$ be a root of $f(x)-t=0$.  Then, for any other
root $z$ of $f(x)=t$, we must have
\[f(z)-f(y)=0=(z-y)R(z,y),\]
for some polynomial $R(x,y)$, which must therefore contain the minimal
polynomial for $z$ over $\C(y,t)=\C(y)$.  Clearly, $G$ is 2-transitive
if and only if there is an automorphism of $\C(x_1,\ldots,x_n)$ which
fixes $y$, and sends $z$ to any of the roots $x_2$ to $x_n$.  In turn,
this can happen if and only if the polynomial $R$ is irreducible.
\end{list}

\begin{definition}\rm
   The unique polynomial $T_n(x)$ which satisfies
   $T_n(\cos(\theta))=\cos(n\theta)$
   is called the {\it Chebyshev} polynomial of degree $n$.
   Equivalently $T_n((z+z^{-1})/2) = (z^n+z^{-n})/2$.
\end{definition}

From the definition, the Chebyshev polynomial $T_n$ has $n-1$
distinct turning points when $T_n=\pm 1$.
Conversely, it can be shown that any polynomial $T(x)$ with just
two critical values and with all turning points distinct must
be equivalent to $T_n(x)$ for some $n$ after pre- and post- composition
with suitable linear functions.

We would like to thank Peter M\"uller for bringing the following
result to our attention.  We give a proof for completeness.

\begin{theorem}\label{muller}
   Let $f(x)$ be a polynomial of degree $n$ and $G=\mbox{Mon}(f)$,
   then one of the following holds.
   \begin{list}{(\roman{listcount})\hfill}{\usecounter{listcount} \setlength{\leftmargin}{4em} \setlength{\labelwidth}{2em}}
      \item The action of $G$ on the $x_i$ is 2-transitive
      \item The action of $G$ on the $x_i$ is imprimitive
      \item $f$ is equivalent to a Chebyshev polynomial $T_p$ where $p$ is prime.
      \item $f$ is equivalent to $x^p$ where $p$ is prime.
   \end{list}
\end{theorem}

\begin{remark}
In particular, the question of whether $f$ is a composite
polynomial or not, can be solved very simply by considering
whether or not the divided differences polynomial factorizes
or not, having excluded the two exceptional cases above.
\lq\lq Equivalence" refers to pre- and post- composition by linear
functions.
\end{remark}

\proof
From Proposition~\ref{BS2}, we can apply the Burnside-Schur Theorem to show
that the group must be 2-transitive, imprimitive, or a subgroup of
$\mbox{Aff}(\Z_p)$.  In the latter case we note that $n=p$, and every
element of $\mbox{Aff}(\Z_p)$ fixes at most one element of
$\Z_p$.  This means that for every critical value of $f$
there is at most one $x_i$ that remains fixed as we turn around
this value.

Now, suppose $f$ has $r$ distinct critical values,
$t_1,\ldots,t_r$, and $f$ has $r_i$ distinct turning points
associated to the critical value $t_i$.  Let the
multiplicities of the roots of $f'$ at these turning points
be $m_{i,1},\ldots,m_{i,r_i}$.  Since a root of multiplicity
$m_{i,j}$ gives a cycle of order $m$, then for all
$i$ we must have
\begin{equation}
  n-1 \le \sum_{j=1}^{r_i} (m_{i,j}) \le n,
\end{equation}
since at most one of the $x_i$ remains fixed when turning
around each critical value.
Summing these equations over $i$ we obtain
\begin{equation}
  r(n-1) \le \sum_{i=1}^r\sum_{j=1}^{r_i} (m_{i,j}) \le rn.
\end{equation}
But the number of turning points of $f$ counted with multplicity
is just the sum of the $m_{i,j}$, and hence
\begin{equation}\label{sumrn}
   r(n-1) \le (n-1) + \sum_{i=1}^r{r_i} \le rn.
\end{equation}
Since the sum of the $r_i$ is at most $n-1$ we must have
$r \le 2$.

If $r=1$, then \eqref{sumrn} shows that $r_1 = 1$, and therefore
$f(x)$ must have a root of multiplicity $n$.  This is just Case (iv),
noting that $n$ is prime.

If $r=2$ we need $n-1 \le r_1 + r_2 \le n+1$.  But since $r_i$ can
be no more than $n/2$ this means that both $r_i$ lie between $(n-1)/2$ and $n/2$.
This implies that every turning point must have multiplicity 1 and
the polynomial must be Chebyshev with $n$ prime.

\section{Proof of the $0$-dimensional theorems}

Having dealt with the preliminaries, the proof of the $0$-dimensional
theorems are straight forward.

\subsection{Proof of Theorem~\ref{center}}

Let $\delta(t)=x_i(t)-x_j(t)$ be a simple cycle, and
let ${\cal{F}}\subset \C(x)$ denote the field of all rational
functions $R\in \C(x)$ for which $R(x_i(t))=R(x_j(t))$.
Clearly $\C \subset {\cal{F}}$, and from
the hypothesis of the theorem, $f$ and $\omega$ both
lie in $\cal{F}$, so it contains at least one non-trivial
element.  However, $x$ does not lie in ${\cal{F}}$, so we
have
\[ \C \varsubsetneq {\cal{F}} \varsubsetneq \C(x). \]

By L\"uroth's theorem, there exists a non-trivial
rational function $h(x)$, necessarily of degree greater
than one, such that $\cal{F}$ is generated
by $h(x)$.  In particular, $f$ and $\omega$ are rational
functions of $h(x)$.  However, by Lemma~\ref{lem5.0}
this implies that after a M\"obius transformation,
the generator $h(x)$ can be taken to be a polynomial,
and $f$ and $\omega$ are polynomials of $h(x)$.  Since $h(x)$
lies in $\cal{F}$, $h(x_i)=h(x_j)$ and we are done.

\subsection{Proof of Theorem~\ref{0dimmon}}

Let $\delta(t)=x_i(t)-x_j(t)$ be a simple cycle, and
let $G = \mbox{Mon}(f)$ be the monodromy group of $f$.
Consider the graph with vertices $x_1,\ldots,x_n$ and
whose edges consist of all pairs $\{x_r,x_s\}$ for which
there exists a $\sigma$ in $G$ such that
$\{\sigma(x_i),\sigma(x_j)\} = \{x_r,x_s\}$.
Every vertex lies on at least one edge, since $G$ is transitive.

If two roots $x_r$ and $x_s$ lie in a connected component of
the graph, then it is clear that we can obtain $x_r-x_s$ as
a sum of terms of the form $\pm(\sigma_k(x_i)-\sigma_k(x_j))$.
Thus, if the monodromy of the cycle $\delta(t)$ does not
generate the whole of $H_0(f^{-1},\Z)$, then there must be
more than one connected component of the graph.  Let $S$ be
the connected component of the graph which contains $x_i$ and
$x_j$.

Each element of $G$ gives an automorphism of the graph in
a natural way.  Take $H$ to be a subgroup of $G$ which sends
$S$ to itself.  Clearly $H$ contains $G_{x_i}$ and also some
element, $\sigma_{ij}$, which sends $x_i$ to $x_j$.
However, if the graph is not
connected, $H$ is strictly smaller than $G$.  Thus, from
the proof of Proposition \ref{monthms}, $f(x)$ is decomposable
with $f(x)=g(h(x))$, where $h(x)$ generates the fixed field
of $H$.  Finally, $h(x_j)=\sigma_{ij}(h(x_i))=h(x_i)$, since
$\sigma_{ij}$ lies in $H$.

\section{The tangential center focus problem in the hyperelliptic case}

\begin{proposition}\label{preparation}
  Let $\omega$ be a polynomial 1-form, and $F(x,y)=y^2+f(x)$ a
  polynomial.  Then, there exists polynomials $A,B\in \C[x,y]$
  and $g\in C[x]$, such that
  \[\omega = A\,dF + dB + yg\,dx.\]
\end{proposition}

\proof  First, it is clear that we can write
$\omega = dB'(x,y) + A'(x,y)\,dx$ for
an appropriate choice of polynomials $A',B'\in\C[x,y]$.
Then, using inductively the identity
\[ a(x)y^{n+2}\,dx = \frac{(n+2)}{2}A(x)f'(x)y^{n}\,dx
+ d(A(x)y^{n+2}) - \frac{(n+2)}{2}A(x)y^n\,dF,  \]
where $A(x)$ is a primitive of $a(x)$, we obtain
the result.
\smallskip

We now prove Theorem~\ref{hyperellipticthm}: from
Proposition~\ref{preparation} we need only consider
the case $\omega = y k(x)\,dx$.
\smallskip

Without loss of generality, we can assume that the tangential
center is at the origin, and that we have scaled $x$ so that
$f(x) = x^2 + O(x^3)$.  We define an analytic function $X$ to
be the unique solution of the equation
\[X^2 = f(x), \qquad X = x + O(x^2).\]

With respect to the coordinates $(X,y)$, the vanishing cycles
can be reparameterized to give the circles $X^2+y^2=t$.
Furthermore, our Abelian integral \eqref{ab} becomes
\[
   \int_{\delta(t)} y k(x)\,dx = \int_{X^2+y^2=t} y m(X)X\,dX,
\]
where
\[
    m(X(x)) = \frac{2k(x)}{f'(x)}.
\]
Clearly this integral vanishes for small values of $X$ if and
only if $k(0)=0$, and $m(X)$ is even in $X$.  That is,
\[
    m(X(x)) = \phi(X(x)^2) = \phi(f(x)),
\]
for some analytic function $\phi$.  Thus,
\[
    2k(x) = f'(x) \phi(f(x)).
\]

Taking $\Phi$ to be a primitive of $\phi$ with $\Phi(0)=0$, and
$K$ to be a primitive of $2k$ with $K(0)=0$, we obtain
\[
    K(x) = \Phi(f(x)).
\]

Now, this means that $K(x)$ vanishes with respect to the
cycle defined by $f(x) = X^2$, and by the proof of
Theorem \ref{center} in the previous section, we must have
both $K$ and $f$ to be composites of a common polynomial
$h(x)$: $K = r\circ h$ and $f=g\circ h$.

Finally, taking $\pi(x,y) = (h(x),y) = (z,y)$, we find
\[
    \omega = yk(x)\,dx = yr'(h(x))h'(x)\,dx = \pi^*(yr'(z)\,dz).
\]
This concludes the theorem once we note that the vanishing
cycle is pushed forward to a cycle homotopic to zero in the
$(z,y)$ coordinates.  This is true as they lie on a family of parabolas
$X+y^2 = t$.

\section{The monodromy problem in the hyperelliptic case}

We consider the level curves of the hamiltonian $H = y^2 - f(x) = t$
as a two sheeted covering of the complex plane $\C$ given by
projection onto the $x$-axis.  The sheets ramify at the roots of $f(x)=t$.
Taking $\Sigma$ to be the set of critical points as above, we let $t$
vary in $\C\setminus\Sigma$,
and follow the effect on the homology group $H_1(F^{-1}(t),\Z)$.
We wish to relate this group to the monodromy group of the polynomial
$f(x)$.  As $x$ tends to infinity along the positive real axis, we can
distinguish the two sheets as ``upper" and ``lower" depending on whether
$y=\pm x^{n/2}$.  We let $\tau$ denote the deck transformation which takes
$y$ to $-y$ fixing $x$.

\begin{figure}[!h]
\begin{center}
\includegraphics[height=7cm]{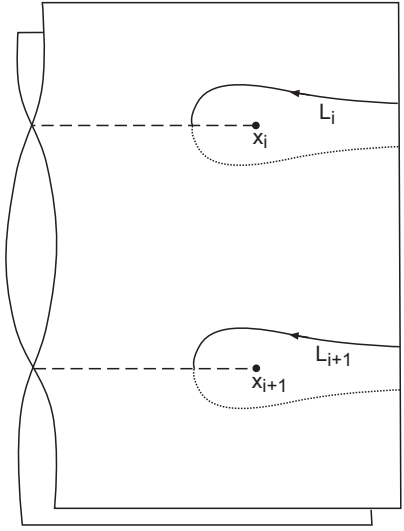}
\caption{The loops $L_i$} \label{modulus0}
\end{center}
\end{figure}

Let $H^c_1(F^{-1}(t),\Z)$ represent the homology with closed support of $F^{-1}(t)$
over $\Z$.  This can be obtained from $H_1(F^{-1}(t),\Z)$ by adding unbounded
closed curves.  Let $x_i(t)$ be the roots of $f(x)=t$.
Generically, the $x_i$ will having distinct imaginary parts, and so
any closed path in $\C\setminus\Sigma$ can be deformed so that only two of
the $x_i$'s have the same imaginary part at the same time.  In other words,
we can decompose every element of $\mbox{Mon}(f)$ as a number of swaps of
$x_i$'s with neighboring real values.

Suppose that the $x_i$ are initially numbered in order of decreasing
imaginary part for a value of $t$ close to zero.
We let $L_i$ represent the path from infinity (from
the direction of the positive real axis) on the upper sheet,
turning around $x_i$ in the positive direction and returning to
infinity on the lower sheet.
Clearly $\tau(L_i)+L_i$ is homotopic to zero, and so the $L_i$ generate $H^c_1(F^{-1}(t),\Z)$.
Furthermore, the elements $L_i-L_{i+1}$ generate $H_1(F^{-1}(t),\Z)$.

The effect of a swap of $x_i$ and $x_{i+1}$ is to take
$L_{i+1}$ to $L_i$ and $L_i$ to $2L_i-L_{i+1}$.
This is a little too complex to analyze in general, except
for very specific systems.  Instead we shall work for the moment
over $\Z_2$.  That is, we consider the images of the $L_i$ in
$H^c_1(F^{-1}(t),\Z_2)$ and $H^c_1(F^{-1}(t),\Z_2)$.

Working modulo 2 means that a swap of $x_i$ and $x_{i+1}$
takes $L_{i+1}$ to $L_i$ and $L_i$ to $L_{i+1}$.  That is,
the action of $\mbox{Mon}(f)$ on the $L_i \pmod{2}$ is
exactly the same as the action on the $x_i$.

We now apply the results of Theorem~\ref{muller}
in order to prove Theorem~\ref{hyperellipticmon}.
According to Theorem~\ref{muller} we only need to
consider four cases.

We shall show below that the Cases (i) and (ii) of
Theorem~\ref{muller} correspond to Cases (i)
and (ii) of Theorem~\ref{hyperellipticmon}.
Case (iii) can be dealt with by adapting the proof for case (i),
and in Case (iv) the Hamiltonian does not have a Morse point,
and hence there are no tangential centers.

\medskip

\noindent{\bf Case (i)/(iii)} If the monodromy group of
$f$ is 2-transitive then we can find a transformation
which takes any two $x_i$'s to any other two.
Since, working modulo two, the action on the loops
$L_i$ is the same as the action on the $x_i$, we can
find an element of the monodromy group which takes
$L_i-L_{i+1}$ to $L_{j}-L_{j+1}$ modulo 2 for all $i$ and $j$.

Now, the vanishing cycle $\delta(t)$ occurs at the coalescence
of two of these $x_i$'s and so must correspond to one of
the $L_k-L_{k+1}$ for some $k$.
Thus, there exist paths $\ell_i$ in $\C\setminus\Sigma$
such that
\[ \sigma(\ell_i)\delta(t) = L_i-L_{i+1}  \pmod{2}, \]
for all $i$.

Now let $N = 2\lfloor (n-1)/2\rfloor$.  Then
$L_i-L_{i+1}$ form a basis of $H_1(F^{-1}(t),\Z)$.
From the discussion above, we have
\[\left(
    \begin{array}{c}
      \sigma(\ell_1)\delta(t) \\
      \sigma(\ell_2)\delta(t) \\
      \vdots \\
      \sigma(\ell_N)\delta(t) \\
    \end{array}
  \right)
  =
  A
  \left(
    \begin{array}{c}
      L_1-L_2 \\
      L_2-L_3 \\
      \vdots \\
      L_N-L_{N+1} \\
    \end{array}
  \right),
\]
where the matrix $A$ reduces to the identity matrix if we
reduce modulo 2.  In particular, $A$ is invertible, and we can express
the basis of $H_1(F^{-1}(t),\Z)$ as sums of the $\sigma(\ell_i)\delta(t)$
with coefficients in $\Q$.  That is, $\delta(t)$ generates $H_1(F^{-1}(t),\Q)$.
This gives us Case (i) of Theorem~\ref{hyperellipticmon}.

Note that in Case (iii), the monodromy group is not 2-transitive.  However
it is still possible to generate each of the $L_i-L_j$ over $\Z_2$ as a sum
of $\sigma(\ell_k)\delta(t)$.  This follows directly from the $\Z_p$ action
of Proposition \ref{BS2} on the roots, and hence on the $L_i$ over $\Z_2$.
The proof then proceeds as above.

\medskip

\noindent{\bf Case (ii)}  In this case, we can assume that
$\mbox{Mon}(f)$ is imprimitive (but not $2$-transitive) on
the roots of $f(x)=t$.  From Theorem \ref{0dimmon},
the function $f(x)$ decomposes as $f(x) = g(h(x))$ with $h(x_i)=h(x_j)$.
This gives us Case (ii) of Theorem~\ref{hyperellipticmon} once we note
that the vanishing cycle for $t$ close to the bifurcation value is
pushed forward to a cycle homotopic to zero via $h$.

\medskip

This completes the proof of Theorem~\ref{hyperellipticmon}.


\section{Generalized monodromy, tangential problems and Chebyshev polynomials}
\label{generalizedmon}

In this final section we would like to consider the possibility
of generalizing the tangential center focus problem
or monodromy problem to the case where the cycle $\delta(t)$ lies in
$H_1(F^{-1}(t),\C)$.  We will show that the Chebyshev polynomials
give counterexamples to both Theorems \ref{hyperellipticthm} and
\ref{hyperellipticmon} in this case.  That is, in the Chebyshev case
there are non-trivial subspaces of $H_1(F^{-1}(t),\C)$ which are
invariant under the monodromy, and one can choose a cycle $\delta(t)$
in this subspace and a $1$-form $\omega$ so that $\omega$ is neither
relatively exact, nor $f(x)$ decomposable.

We make the following conjecture along the lines of Theorem~\ref{hyperellipticmon}.

\begin{conjecture}\label{conj}
  If there exists a non-trivial subspace of $H_1(F^{-1}(t),\C)$
  which is invariant under the monodromy, then either the polynomial
  $f$ decomposes as $f=g\circ h$, or $f$ is equivalent to either $x^p$ or
  the Chebyshev polynomial $T_p$ for some prime $p$.
\end{conjecture}

For completeness, it would be also interesting to investigate,
in analogy with Theorems~\ref{hyperellipticthm} and \ref{hyperellipticmon},
whether any cycles $\delta(t)$ which lie in the invariant subspace of
$H_1(F^{-1}(t),\C)$ and any 1-form $\omega$ for which $\int_{\delta(t)}\omega\equiv 0$
also must factor through $h$ if there is a decomposition.  We do not consider
these questions here.

In the $0$-dimensional case, Conjecture \ref{conj} is in fact a theorem.
From Propositions \ref{monthms} and \ref{muller}, if the monodromy
is not imprimitive (and hence the polynomial $f$ decomposes),
it is either equivalent to $x^p$ or $T_p(x)$ for some prime $p$, or
the monodromy group is $2$-transitive.  In the latter case, it is
a classical result (\cite{hall}, p281) that the permutation representation
over $\C$ of a $2$-transitive group decomposes into two irreducible
sub-representations.  One is the trivial one, and one the space
$\{\sum k_i x_i|k_i\in \C, \sum k_i = 0\}$, which is just our
space $H_0(f^{-1}(t),\C)$.

\begin{theorem}\label{section7thm}
  Let $f(x) = T_p(x)$ be the Chebyshev polynomial for some prime $p>2$, and let $F = y^2 + f(x)$.
  \begin{itemize}
    \item[(i)]  The space $H_0(f^{-1}(t),\C)$ splits into $p-1$ invariant subspaces $W_{e^{2\pi k i/p}}$, $k=1,\ldots, (p-1)/2$.
    Furthermore, if $\delta(t) \in W_{e^{2\pi k i/p}}$, $k=2,\ldots,(p-1)/2$, then there exists a 0-form $\omega$
    such that $\int_{\delta(t)}\omega \equiv 0$, but $\omega$ is not decomposable.

    \item[(ii)] The space $H_0(F^{-1}(t),\C)$ splits into $p-1$ invariant subspaces $V_{e^{2\pi k i/p}}$, $k=1,\ldots, (p-1)/2$.
    Furthermore, if $\delta(t)\in V_{e^{2\pi k i/p}}$, $k=2,\ldots,(p-1)/2$, then there exist a 1-form $\omega$
    such that $\int_{\delta(t)}\omega \equiv 0$, but $\omega$ is not relatively exact.
  \end{itemize}
\end{theorem}

\begin{remark}
  It would be sufficient to consider homology groups with
  coefficients in $\Q(w)$ for some $p$-th root of unity $w$.
\end{remark}

The rest of this section is devoted to proving Theorem \ref{section7thm}.  We assume that $n$ is odd throughout.
\smallskip

Recall that $T_n(x)$ is defined by
\begin{equation}
\label{Tn} T_n(x) = T_n(\cos(\t))=\cos(n\t),\quad \text{where}\quad x=\cos \t
\end{equation}
Clearly $T_n$ has degree $n$, and hence $T_p(x)$ is not decomposable for
$p$ prime.

We prove the theorem by pulling back to $\t$ coordinates.
Let $\tilde F:\C^2\to\C$ be the function
$$
\tilde F(\t,y)=y^2+\cos(n\t)$$ and let $X_{\tilde F}$ be the
Hamiltonian vector field associated to $\tilde F$. The vector field
$X_{\tilde F}$ has infinitely many singular points
$p_\ell=(\frac{\ell\pi}{n},0)$, $\ell\in\Z$. These points are
saddles $s_{2k}=(\frac{2k\pi}{n},0)$ for $\ell=2k$ even and centers
$c_{2k+1}=(\frac{\pi}{n}+\frac{2k\pi}{n},0)$, for $\ell=2k+1$ odd.

For $t\in(-1,1)$, let $\tilde C_{2k+1}$, $k\in\Z$, be the cycle
turning once in the positive direction around the center $c_{2k+1}$.
All cycles $\tilde C_{2k+1}$ vanish for $t=-1$. Similarly, let
$\tilde S_{2k}$ be the complex cycle vanishing at the saddle
$s_{2k}$, for $t=1$. The orientation is chosen by the condition:

$$
(\tilde C_{2i-1},\tilde S_{2i})=(\tilde S_{2i},\tilde C_{2i+1})=1.$$
We denote $\tilde P_\ell$ the cycle $\tilde S_\ell$, for $\ell$ even
or $\tilde C_\ell$, for $\ell$ odd.
 The complex fiber $\tilde
F^{-1}(t)$ can be represented as a two-sheeted Riemann surface
$y=\sqrt{t-\cos(n\t)}$, with a countable number of cuts.
The homology
group of a fiber $H_1(\tilde F^{-1}(t),\Z)$ for
$t\in\C\setminus\{-1,1\}$, is the free abelian group on the set of
cycles $\cup_{i\in\Z}\{C_i,S_i\}$.

The flow of the gradient vector field of $\tilde F$ allows us to
define a compact support fibration on
$\C^2\setminus (\tilde F^{-1}(-1)\cup \tilde F^{-1}(1))$. That is,
for any $t_0\in\C\setminus\{-1,1\}$, and any compact $K$ in
$F^{-1}(t_0)$, there exists a neighborhood $U$ of
$t_0\in\C\setminus\{-1,1\}$ and an embedding $\Phi:U\times K\to
\C^2\setminus\{-1,1\}$, such that $\Phi(t_0,p)=p$ and
$F\circ\Phi(t,p)=t$, for any $t\in U$. Moreover, the trivialization
$\Phi$ is well defined up to an isotopy which is identity on $K$ and
preserves the fibers.

The existence of the compact support fibration enables the definition
of the monodromy acting on $H_1(\tilde F^{-1}(t),\Z)$. In fact, by
the Picard-Lefschetz formula, it follows:
\begin{equation}\label{moneqn}
\begin{array}{rclrcl}
\tilde M_1(\tilde C_{2i+1})&=&\tilde C_{2i+1}+\tilde S_{2i}-\tilde S_{2i+2},
&\tilde M_1(\tilde S_{2i})&=&\tilde S_{2i},\\
\tilde M_{-1}(\tilde C_{2i+1})&=&\tilde C_{2i+1}, &\tilde
M_{-1}(\tilde S_{2i})&=&\tilde S_{2i}+\tilde C_{2i-1}-\tilde
C_{2i+1}.
\end{array}
\end{equation}

Consider the mapping $\cos:\C\to\C$ and denote $\Pi=\cos\times Id$,
then
$$\begin{array}{rcl}
(\t,y)\in\C^2&\buildrel{\Pi}\over{\longrightarrow}&(x,y)\in\C\setminus\{-1,1\}\times \C\\
\searrow \tilde  F&&\swarrow F\\
&\C\setminus\{-1,1\} &\end{array}
$$
Let
$$
P_\ell=\Pi_*(\tilde P_\ell),\quad S_{2\ell}=\Pi_*(\tilde
S_{2\ell}),\quad C_{2\ell+1}=\Pi_*(\tilde C_{2\ell+1}).
$$

The map $\cos:\C\setminus\pi\Z\to\C\setminus\{-1,1\}$ is a covering with covering group
$G=\Z_2*\Z_2 = D_\infty$ generated by two transformations of order
$2$: $a(\t)=-\t$ and $b(\t)=2\pi-\t$. The composition $b\circ a$ is
the translation $\t\mapsto 2\pi+\t$, which we denote $T$.  We take $a$
and $T$ as the generators of $D_\infty$, with $Ta=aT^{-1}$.

The map $\Pi:(\C\setminus\pi\Z)\times\C\to(\C\setminus\{-1,1\})\times\C$
is a covering with the same covering group $G=D_\infty$ generated
by the two transformations $a\times id$, and $T\times id$.

The action of the group $G=D_\infty$ on the cycles
$\tilde P_\ell$ (i.e. $\tilde C_{2k+1}$ or $\tilde S_{2k}$) is given
by
$$
\begin{array}{rcl}
T\times id(\tilde P_\ell)&=&\tilde P_{\ell+2n},\\
 a\times id(\tilde P_{\ell})&=&-\tilde P_{-\ell},
\end{array}$$

We let $H^c_1(\tilde F^{-1}(t),\C)$ be the homology with closed support of $\tilde F^{-1}(t)$ with complex
coefficients.  An element of $H^c_1(\tilde F^{-1}(t),\C)$ is of the form $C=\sum_{\ell\in\Z}
z_\ell \tilde P_\ell$.

We define the action of the covering group $G$ on $H^c_1(\tilde
F^{-1}(t),\C)$ as follows: $$ g(C)=g(\sum_{\ell\in\Z} z_\ell
P_\ell)=\sum_{\ell\in\Z} z_\ell g(P_\ell).$$

Let $H_1^G(\tilde F^{-1}(t),\C)$ be the subspace of $H^c_1(\tilde
F^{-1}(t),\C)$ consisting of elements of $H_1(\tilde F^{-1}(t),\C)$ invariant
under the action of the group $G$.  The monodromy operators $\tilde M_\sigma$, $\sigma=\pm 1$
extend naturally to $H_1(\tilde F^{-1}(t),\C)$.

The space $H_1^G(\tilde F^{-1}(t),\C)$ is the $\C$-vector space generated
by $\sum_{g\in G} g(P_{\ell})$, $\ell\in\Z,$ and the extended
monodromy $\tilde M_\sigma$ preserves it.

Let $\Pi^*:H_1(F^{-1}(t),\C)\to H^c_1(\tilde F^{-1}(t),\C)$ be the pullback
via the map $\Pi$, then $\Pi^*$ gives an isomorphism of $H_1(F^{-1}(t),\C)$
onto $H_1^G(\tilde F^{-1}(t),\C)$ with inverse
$\Pi' = (\Pi^*)|_{H_1^G(\tilde F^{-1}(t),\C)}^{-1}$ from $H_1^G(\tilde F^{-1}(t),\C)$ to $H_1(F^{-1}(t),\C)$.

Let $M_\sigma$ and $\tilde M_\sigma$, $\sigma=\pm1$,
be the monodromies corresponding to turning around $\sigma=\pm1$, as
given in (\ref{m}). Then the following diagram is commutative
\begin{equation}
\begin{array}{rcl}\label{m}
H_1^G(\tilde
F^{-1}(t),\C)&\buildrel{\Pi'}\over{\rightarrow}&H_1(F^{-1}(t),\C)\\
\tilde M_\sigma\downarrow&& \downarrow M_\sigma\\
H_1^G(\tilde
F^{-1}(t),\C)&\buildrel{\Pi'}\over{\rightarrow}&H_1(F^{-1}(t),\C)\\
\end{array}.
\end{equation}

Let $w$ be an $n$-th rooth of unity.
The vectors $\sum w^\ell \tilde S_{2\ell}$ and  $\sum
w^\ell \tilde C_{2\ell+1}$, are clearly invariant by the
translation $T$, but not by $a$. Taking
\begin{equation}\label{stc}
\begin{array}{rcccl}
\tilde S_w&=&(1+a)\sum w^\ell
\tilde S_{2\ell}&=&\sum_{\ell=-\infty}^{\infty}(w^\ell-w^{-\ell})\tilde S_{2\ell},\\
\tilde C_w&=&(1+a)\sum w^\ell \tilde
C_{2\ell+1}&=&\sum_{\ell=-\infty}^{\infty}(w^\ell-w^{-\ell-1}) \tilde C_{2\ell+1},\\
\end{array}
\end{equation}
we therefore obtain elements of $H_1^G(\tilde F^{-1}(t),\C)$.

We let $S_w$ and $C_w$ in $H_1( F^{-1}(t),\C)$ represent the images of
$\tilde S_w$ and $\tilde C_w$ under $\Pi'$. That is
\begin{equation}\label{sc}
\begin{array}{rcl}
S_w&=&\Pi'\tilde S_w=\sum_{\ell=0}^{n-1}(w^\ell-w^{-\ell})S_{2\ell},\\
C_w&=&\Pi'\tilde C_w=\sum_{\ell=0}^{n-1}(w^\ell-w^{-\ell-1})
C_{2\ell+1}.
\end{array}
\end{equation}

By direct substitution from \eqref{moneqn} we can calculate the variation, $\tilde \V_{t_0}$, $t_0=\pm1$,
on $H_1^G(\tilde F^{-1}(t),\C)$ around $t_0=\pm1$.  Due to \eqref{m}, these calculations push forward
to $H_1(F^{-1},\C)$ via $\Pi'$, to obtain
$$
\begin{array}{rcl}
\tilde \V_1(\tilde C_w)&=&(1+w^{-1})\tilde S_w,\\
\tilde \V_1(\tilde S_w)&=&0,\\
\tilde \V_{-1}(\tilde C_w)&=&0,\\
\tilde \V_{-1}(\tilde S_w)&=&(-1+w)\tilde C_w
\end{array}
\qquad \buildrel{\Pi'}\over{\Rightarrow} \qquad
\begin{array}{rcl}
\V_1(C_w)&=&(1+w^{-1})S_w,\\
\V_1(S_w)&=&0,\\
\V_{-1}(C_w)&=&0,\\
\V_{-1}(S_w)&=&(-1+w)C_w.
\end{array}
$$

We denote $\tilde V_w=Span(\tilde
C_w,\tilde S_w)\subset H_1^G(\tilde F^{-1}(t),\C)$ and
$V_w=\Pi' \tilde V_w\subset H_1(F^{-1}(t),\C).$
The spaces $V_w\subset H_1(F^{-1}(t),\C)$ are invariant under the action
of the monodromy group $\cal M$ of the fibration given by $F$.
Moreover, for $n$ odd,
$$H_1(F^{-1}(t),\C)=\oplus_{w^n=1, Im(w)>0}V_w.$$

\begin{proposition}\label{example}
Let $\delta=S_w$ or $\delta=C_w$  be the family of cycles in
$H_1(F^{-1}(t),\C)$, given by (\ref{sc}), for $w=e^{\frac{2k\pi i}{n}}$,
$k=2,\ldots,(n-1)/2$ and let $\omega=y\,dx$. Then
$$
\int_\delta\omega\equiv0,$$ but the form $\omega$ is not relatively
exact.
\end{proposition}

\begin{proof}
Let $w=\xi^k$, $k=2,...,(p-1)/2$, where $\xi=e^{\frac{2i\pi}{n}}$.

Consider first the case $\delta=S_w$. Let $I=\int_{\tilde S_0}y\cos
\t\, d\t$. We calculate $I_{2\ell}=\int_{\Pi_*(\tilde
S_{2\ell})}y\,dx=-\int_{\tilde S_{2\ell}}y\sin \t\, d\t.$ We make a change
of coordinates $\t\mapsto \t+\frac{2\ell\pi}{n}$. This gives
$I_{2\ell}=-\cos\frac{2\ell\pi}{n}\int_{\tilde S_0}y\sin
\t\, d\t-\sin\frac{2\ell\pi}{n}\int_{\tilde S_0}\cos \t\, d\t.$ The first
integral vanishes, giving

\begin{equation}\label{i2l}
I_{2\ell}=-\sin\frac{2\ell\pi}{n}I=-\frac{\xi^\ell-\xi^{-\ell}}{2i}I.
\end{equation}
This gives
\begin{equation}\label{sw}
\begin{array}{rl}
\int_{S_w}y\,dx&=\sum_{\ell=0}^{n-1}
(w^\ell-w^{-\ell})I_{2\ell}=-I\sum_{\ell=0}^{n-1}
(\xi^{k\ell}-\xi^{-k\ell})\frac{\xi^\ell-\xi^{-\ell}}{2i}\\&=-\frac{I}{2i}\sum_{\ell=0}^{n-1}
(\xi^{(k+1)\ell}-\xi^{(k-1)\ell}-\xi^{(-k+1)\ell}+\xi^{-(k+1)\ell})=0.
\end{array}
\end{equation}
The last equality holds as each of the four sums which appear
vanishes.
 Consider now the case $\delta=C_w$.
Denote $T_{-\pi/n}(\tilde C_1)$ the transport of the translation of
the cycle $\tilde C_1$ by $-\pi/n$, thus giving a cycle centered at
the origin. Let $J=\int_{T_{-\pi/n}(\tilde C_1)}y\cos \t\, d\t.$ We
calculate $J_{2\ell+1}=\int_{\Pi_*(\tilde
C_{2\ell+1})}y\,dx=-\int_{\tilde C_{2\ell+1}}y\sin \t\, d\t.$ We make the
change of coordinates $\t\mapsto \t+\frac{(2\ell+1)\pi}{n}$. This
gives
$J_{2\ell+1}=-\cos\frac{(2\ell+1)\pi}{n}\int_{T_{-\pi/n}(\tilde
C_1)}y\sin \t\, d\t-\sin\frac{(2\ell+1)\pi}{n}\int_{T_{-\pi/n}(\tilde
C_1)}y\cos \t\, d\t=-\sin\frac{(2\ell+1)\pi}{n}J.$ That is
\begin{equation}\label{j2l+1}
J_{2\ell+1}=-\sin\frac{(2\ell+1)\pi}{n}J=-\frac{\xi^{\ell+1/2}-\xi^{\ell-1/2}}{2i}J.
\end{equation}
\begin{equation}\label{cw}
\int_{C_w}ydx=\sum_{\ell=0}^{n-1}(w^{\ell}-w^{-\ell-1})J_{2\ell+1}=
-\frac{J}{2i}\sum_{\ell=0}^{n-1}(\xi^{k\ell}-\xi^{-k\ell-k})(\xi^{\ell+1/2}-\xi^{\ell-1/2})
=0,
\end{equation}
similarly to (\ref{sw}).

Note that it is obvious that the form $\omega=ydx$ is not relatively
exact since for instance $\int_{C_1}y\,dx\not=0$ is the non-zero area
bounded by $C_1$.
\end{proof}
\medskip

This completes the proof of part (ii) of the Theorem \ref{section7thm}.
We now prove the statement in part (i).
\medskip

Let
$\t_0^\pm=\pm\frac{\arccos t}{n},$
$\t^\pm_\ell=\t^\pm_0+\frac{2\pi\ell}{n}$,
$x_\ell^\pm=\cos(\t_\ell^\pm)$. Note that
$x_\ell^\pm=\cos(\t^\pm_0+\frac{2\pi\ell}{n})=\cos
\t_0^\pm\cos\frac{2\pi\ell}{n}-\sin \t_0^\pm\sin\frac{2\pi\ell}{n}$.
Hence,

\begin{equation}\label{xell}
x_\ell^+-x_\ell^-=-2\sin \t_0^+\sin\frac{2\pi\ell}{n}=i\sin
\t_0^+(\xi^\ell-\xi^{-\ell}).
\end{equation}

Let $$\delta_{2\ell}(c)=x^+_\ell(c)-x^-_\ell(c),\quad
\ell=0,\ldots,n-1,$$
 be  the families of simple cycles of the
Chebyshev polynomial $T_n$, and let
$$
\delta_w(t)=\sum_{\ell=0}^{n-1}(w^\ell-w^{\ell-1})\delta_{2\ell}\in H_0(f^{-1}(t),\C)
$$
where $w=e^{\frac{2k\pi i}{n}}$,
$k=1,\ldots,(n-1)/2$. Taking $W_w$ as the subspace of $H_0(f^{-1}(t),\C)$
spanned by $\delta_w$, it is clear that $W_w$ is invariant under the monodromy,
and $$H_0(f^{-1}(t),\C)=\oplus_{w^n=1, Im(w)>0}W_w.$$

\begin{proposition}\label{0dim}
 The $0$-dimensional Abelian integral $I=\int_{\delta}\omega$
vanishes identically, for the cycle $\delta=\delta_w$, with
$w=e^{\frac{2k\pi i}{n}}$, $k=2,\ldots,(n-1)/2$ and $\omega(x) = x$,
but the $0$-form $\omega$ is not relatively exact.
\end{proposition}

\begin{proof}
The proof is similar to the proof of the previous theorem. In fact it
is simpler. The simple cycles $\delta_{2\ell}$ entering in the
definition of the cycle $\delta$ corresponds to the ramification
points around which the cycle $S_{2\ell}$ turns. We have
$$\int_\delta\omega=\sum_{\ell=0}^{n-1}(w^\ell-w^{-\ell})(x_\ell^+-x_\ell^-)=
i\sin
\t_0^+\sum_{\ell=0}^{n-1}(\xi^{k\ell}-\xi^{-k\ell})(\xi^\ell-\xi^{-\ell})=0.
$$
On the other hand
$\int_{\delta_\ell}\omega=x_\ell^+-x_\ell^-=\frac{2\arccos
t}{n}\not=0$, so $\omega$ is not relatively exact.
\end{proof}


\end{document}